\documentclass[10pt,twoside]{siamart1116}

\usepackage[english]{babel}
\usepackage{graphicx,epstopdf,epsfig}
\usepackage{amsfonts,epsfig,fancyhdr,graphics, hyperref,amsmath,amssymb}
\usepackage{amsfonts} 
\usepackage{mathrsfs}
\usepackage[all,cmtip]{xy}

\newcommand*{\Scale}[2][4]{\scalebox{#1}{$#2$}}

\newcommand{\Hilbert}{Hilbert }
\newcommand{\Hil}{\mathscr{H}}
\newcommand{\C}{\mathbb{C}}
\newcommand{\Hermitian}{self-adjoint }

\newcommand{\GG}{G}
\newcommand{\PP}{\widehat{P}}
\newcommand{\LL}{L}
\newcommand{\II}{\widehat{I}}
\newcommand{\GA}{\mathit{\Gamma}}
\newcommand{\SI}{\mathit{\Sigma}}
\newcommand{\HH}{H}
\newcommand{\KS}{K}
\newcommand{\Ks}{\mathrm{K}}

\newcommand{\Gg}{\mathrm{G}}
\newcommand{\Pp}{\widehat{\mathrm{P}}}
\newcommand{\Ll}{\mathrm{L}}
\newcommand{\Ii}{\mathrm{I}}
\newcommand{\Ga}{\Gamma}
\newcommand{\Hh}{\mathrm{H}}
\newcommand{\Vv}{\mathrm{V}}
\newcommand{\La}{\Lambda}
\newcommand{\scale}{.54}

\newcommand{\be}{\begin{align}}
\newcommand{\ee}{\end{align}}

\newcommand{\HE}{Namie of Handling Editor}
\newcommand{\DoS}{Month/Day/Year}
\newcommand{\DoA}{Month/Day/Year}
\newcommand{\CA}{Name of Corresponding Author}

\newcommand{\Names}{Matteo Polettini}
\newcommand{\Title}{Oblique projections on metric spaces}

\renewcommand{\theequation}{\thesection.\arabic{equation}}
\renewtheorem{theorem}{Theorem}[section]

\begin{document}

\bibliographystyle{plain}

\setcounter{page}{1}

\thispagestyle{empty}

 \title{\Title\thanks{Received
 by the editors on \DoS.
 Accepted for publication on \DoA. 
 Handling Editor: \HE. Corresponding Author: \CA}}

\author{
Matteo Polettini\thanks{University of Luxembourg, Facult\'e des Sciences, de la Technologie et de la Communication
162 A, avenue de la Faiencerie L-1511 Luxembourg (Grand Duchy of Luxembourg)
(matteo.polettini@uni.lu).}}

\markboth{\Names}{\Title}

\maketitle

\begin{abstract}
It is known that complementary oblique projections $P_0 + P_1 = I$ on a \Hilbert space $\Hil$ have the same standard operator norm $\|P_0\| = \|P_1\|$ and the same singular values other than $0$ and $1$. We generalize to \Hilbert spaces endowed with a positive-definite metric $\GG$ on top of the scalar product, proving that the volume elements (pseudodeterminants $\det_+$) of the metrics $\LL_0,\LL_1$ induced by $\GG$ on the complementary oblique subspaces $\Hil_0 \oplus \Hil_1 = \Hil$, and of those $\GA_0,\GA_1$ induced on their algebraic duals, satisfy
\begin{align}
\frac{\det_+ \LL_1}{\det_+ \GA_0} = \frac{\det_+ \LL_0}{\det_+ \GA_1} = {\det}_+ \GG. \nonumber
\end{align}
Furthermore, we break down this result to spectra, proving that operators $\sqrt{\GA_0 \LL_0}$ and $\sqrt{\LL_1 \GA_1}$ have the same singular values other than $0$ and $1$. We connect the former result to a well-known duality of weighted spanning-tree polynomials in graph theory.
\end{abstract}




\section{Introduction}

Oblique projections and their singular values have been extensively studied in Ref.\,\cite{lewko} and reviewed in Ref.\,\cite{szyld}. In a previous work or ours (see Ref.\,\cite{polettini1}) we employed these ideas to give an algebraic interpretation of the decomposition of the edge space of a graph in the basis of so-called ``cycles'' and ``cocycles''. The main objective of this work is to extend these latter results to arbitrary projection operators, and to graphs that carry positive weights along their edges. Since such weights can be interpreted as a metric, ultimately the generalization presented here is to arbitrary oblique projections on metric \Hilbert spaces.

Following Gower's suggestion \cite{gowers}, we introduce our main results by self-explanatory examples.

\begin{paragraph}{Main results} Consider the matrix
\begin{eqnarray}
\Pp_0 = \left(
\begin{array}{cccc}
 1 & 1 & 0 & 0 \\
 0 & 0 & 0 & 0 \\
 0 & 1 & 0 & 1 \\
 0 & 1 & 0 & 1 \\
\end{array}
\right).
\end{eqnarray}
(The hat denotes operatos, to distinguish them from bilinear forms, see later.)
Clearly $\Pp_0^2 = \Pp_0$ and $\Pp_0 \neq \Pp_0^\dagger$, therefore $\Pp_0$ is an oblique projection. Let $\Pp_1 := \Ii - \Pp_0$ be its complement. Furthermore consider the positive-definite symmetric matrix:
\begin{equation}
\Gg = \left(
\begin{array}{cccc}
 1 & 0 & 0 & 0 \\
 0 & 2 & 1 & 0 \\
 0 & 1 & 2 & 0 \\
 0 & 0 & 0 & 1 \\
\end{array}
\right).
\end{equation}
We define the induced metrics as $\Ll_0 := \Pp_0^\dagger \Gg \Pp_0$, $\Ga_0 := \Pp_0 \Gg^{-1} \Pp_0^\dagger$, $\Ll_1 := \Pp_1^\dagger \Gg \Pp_1$, $\Ga_1 := \Pp_1 \Gg^{-1} \Pp_1^\dagger$ (we omit giving the explicit formulas). Letting $\det_+$ denote the product of the nonvanishing eigenvalues of a matrix, we obtain
\begin{align}
\frac{\det_+ \Ll_1}{\det_+ \Ga_0} = \frac{\det_+ \Ll_0}{\det_+ \Ga_1} = {\det}_+ \Gg = 3. \nonumber
\end{align}
Furthermore, let us define the following matrices:
$\Ks_0 := \Ga_0^{1/2} \Ll_0 \Ga_0^{1/2}$, $\Sigma_0 := \Ll_0^{1/2} \Ga_0 \Ll_0^{1/2}$ , $\Ks_1:=  \Ga_1^{1/2} \Ll_1 \Ga_1^{1/2}$, $\Sigma_1 := \Ll_1^{1/2} \Ga_1 \Ll_1^{1/2}$.
Their spectra coincide, and all positive eigenvalues are larger than $1$:
\begin{align}
\sigma(\Ks_0) = \sigma(\Ks_1) = \sigma(\Sigma_0) = \sigma(\Sigma_1) = (\sim 5.36,\sim 1.31,0,0). \nonumber
\end{align}
\end{paragraph}

\begin{paragraph}{Graph polynomials}

Consider the following oriented graph, with positive weights $g_1,\ldots,g_5$ on the edges:
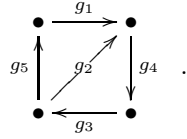
\begin{align}
\begin{array}{c}\xymatrix{\bullet \ar@{->}[r]^{g_1} & \bullet\ar@{->}[d]^{g_4} \\ \bullet \ar@{->}[u]^{g_5}  \ar@{->}[ur]|-{g_2} &  \bullet \ar@{->}[l]^{g_3} }  \end{array}. \nonumber
\end{align}
There exists a standard procedure (see Appendix of Ref.\,\cite{polettini1}) to introduce a basis for the linear space of oriented cycles and that of oriented cocycles (a cocycle is a minimal set of edges whose removal disconnects the graph into two components) starting from an arbitrary spanning tree. 
One such basis of three cocycles and two cycles is
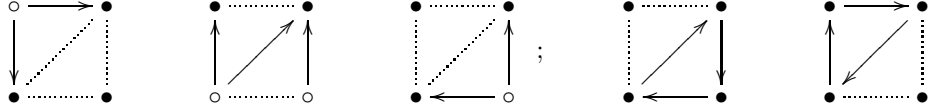
\begin{equation}
\begin{array}{c}\xymatrix{  \circ \ar@{->}[r]  \ar@{->}[d] & \bullet
		\\ \bullet   \ar@{.}[ur]   & \bullet    \ar@{.}[l]    \ar@{.}[u] }\end{array}
\qquad
\begin{array}{c}\xymatrix{\bullet \ar@{<-}[d]  \ar@{.}[r] & \bullet
		\\  \circ \ar@{->}[ur] \ar@{.}[r] & \circ \ar@{->}[u] }\end{array}
\qquad				
\begin{array}{c}\xymatrix{\bullet  \ar@{.}[d] \ar@{.}[r] & \bullet
		\\ \bullet  \ar@{.}[ur] & \ar@{->}[l] \ar@{->}[u] \circ }\end{array} ;
\qquad
\begin{array}{c}\xymatrix{ \ar@{.}[d] \ar@{.}[r] \bullet & \bullet  \ar@{->}[d] \\ \bullet  \ar@{->}[ur]  & \bullet \ar@{->}[l] }\end{array}
\qquad \begin{array}{c}\xymatrix{ \bullet \ar@{->}[r]  \ar@{<-}[d] & \bullet \\ \bullet  \ar@{<-}[ur] & \bullet \ar@{.}[u] \ar@{.}[l] }\end{array}
\nonumber
\end{equation}
with vector representatives
\begin{equation}
c_1 = \left( \begin{array}{c} 1 \\ 0 \\ 0 \\ 0 \\ -1 \end{array} \right) \quad
c_2 = \left( \begin{array}{c} 0 \\ 1 \\ 0 \\ -1 \\ 1 \end{array} \right) \quad
c_3 = \left( \begin{array}{c} 0 \\ 0 \\ 1 \\ -1 \\ 0 \end{array} \right); \quad
c_4 = \left( \begin{array}{c} 0 \\ 1 \\ 1 \\ 1 \\ 0 \end{array} \right) \quad
c_5 = \left( \begin{array}{c} 1 \\ -1 \\ 0 \\ 0 \\ 1 \end{array} \right). \nonumber
\end{equation}
 
Vectors $c_1$, $c_2$ and $c_3$ span the cocycle space, and vectors $c_4$ and $c_5$ span the cycle space. Letting $e_i$ be the $i$-th  Cartesian vector, representing the $i$-th single edge, in Ref.\,\cite{polettini1} it is shown that $P_0 = \sum_{i =4,5} e_i \otimes c_i$ and  $P_1 = \sum_{i =1,2,3} c_i \otimes e_i$ are complementary projections. 

Letting $\Gg = \mathrm{diag}\{g_i\}_{i = 1}^5$, in the basis $\{e_i\}_{i=1}^5$ the only nonvanishing blocks of the induced metrics read
\begin{align}
\Ll_0 & = [c_4,c_5]^T \Gg [c_4,c_5]
 = \left(\begin{array}{cc} g_2 + g_3 + g_4 & - g_2 \\
-g_2 & g_1 + g_2 + g_5 \end{array}\right). \nonumber
\\
\Ga_0 & = [e_4,e_5]^T \Gg^{-1} [e_4,e_5]
 = \left(\begin{array}{cc} g_4^{-1} & 0 \\
0 & g_5^{-1} \end{array}\right) \nonumber \\
\Ll_1 & = [e_1,e_2,e_3]^T \Gg [e_1,e_2,e_3]
= \left(\begin{array}{ccc} g_1 & 0 & 0 \\ 0 & g_2 & 0 \\ 0 & 0 & g_3 \end{array}\right) \nonumber \\
\Ga_1 & = [c_1,c_2,c_3]^T \Gg^{-1} [c_1,c_2,c_3]
 = \left(\begin{array}{ccc} g_1^{-1} + g_5^{-1} & -g_5^{-1} & 0 \\ - g_5^{-1} & g_2^{-1} + g_4^{-1} + g_5^{-1} & g_4^{-1} \\ 0 & g_4^{-1} & g_4^{-1} + g_3^{-1} \end{array}\right) \nonumber
\end{align}
Matrices $\Ll_0$ and $\Ga_1$ are sometimes called Kirchhoff-Symanzik matrices \cite{nakanishi}. We have that
\begin{align}
\det \Ll_0 & =  g_4 g_5 + g_1 g_3 +  g_5 g_3  + g_1 g_4 +  g_2 g_3 + g_5 g_2 +   g_1 g_2 +  g_4 g_2   \nonumber \\
&  =
\Scale[\scale]{\begin{array}{c}
\xymatrix{\bullet \ar@{..}[r] & \bullet\ar@{-}[d] \\ \bullet \ar@{-}[u] \ar@{..}[ur] &  \bullet \ar@{..}[l] }\end{array}
} + 
\Scale[\scale]{\begin{array}{c}
\xymatrix{\bullet \ar@{-}[r] & \bullet\ar@{..}[d] \\ \bullet \ar@{..}[u]  \ar@{..}[ur] &  \bullet \ar@{-}[l] }\end{array}
}  +
\Scale[\scale]{\begin{array}{c}
\xymatrix{\bullet \ar@{..}[r] & \bullet\ar@{..}[d] \\ \bullet \ar@{-}[u]  \ar@{..}[ur] &  \bullet \ar@{-}[l] }\end{array}
} +
\Scale[\scale]{\begin{array}{c}
\xymatrix{\bullet \ar@{-}[r] & \bullet\ar@{-}[d] \\ \bullet \ar@{..}[u]  \ar@{..}[ur] &  \bullet \ar@{..}[l] }\end{array}
} +
\Scale[\scale]{\begin{array}{c}
\xymatrix{\bullet \ar@{..}[r] & \bullet\ar@{..}[d] \\ \bullet \ar@{..}[u]  \ar@{-}[ur]&  \bullet \ar@{-}[l]}\end{array}
} +
\Scale[\scale]{\begin{array}{c}
\xymatrix{\bullet \ar@{..}[r] & \bullet\ar@{..}[d] \\ \bullet \ar@{-}[u]  \ar@{-}[ur] &  \bullet \ar@{..}[l] }\end{array}
} +
\Scale[\scale]{\begin{array}{c}
\xymatrix{\bullet \ar@{-}[r] & \bullet\ar@{..}[d] \\ \bullet \ar@{..}[u]  \ar@{-}[ur] &  \bullet \ar@{..}[l] }\end{array}
} +
\Scale[\scale]{\begin{array}{c}
\xymatrix{\bullet \ar@{..}[r] & \bullet\ar@{-}[d] \\ \bullet \ar@{..}[u]  \ar@{-}[ur] &  \bullet \ar@{..}[l] }\end{array}
}  
\nonumber \\
& = \sum_{\mathcal{T}} \prod_{e \notin T} g_e \nonumber \\
\det \Ga_1 & =\frac{1}{g_1 g_2 g_3}+ \frac{1}{g_2 g_4 g_5} + \frac{1}{g_1 g_2 g_4} + \frac{1}{g_2 g_3 g_5} + \frac{1}{g_1 g_4 g_5} + \frac{1}{g_1 g_3 g_4} + \frac{1}{g_3 g_4 g_5} +\frac{1}{g_1 g_3 g_5}   \nonumber
\\ 
& = \Scale[\scale]{\begin{array}{c}
\xymatrix{\bullet \ar@{-}[r] & \bullet\ar@{..}[d] \\ \bullet \ar@{..}[u]  \ar@{-}[ur] &  \bullet \ar@{-}[l]} \end{array}
} + 
\Scale[\scale]{\begin{array}{c}
\xymatrix{\bullet \ar@{..}[r] & \bullet\ar@{-}[d]  \\ \bullet \ar@{-}[u]  \ar@{-}[ur] &  \bullet \ar@{..}[l] }\end{array}
} + 
\Scale[\scale]{\begin{array}{c}
\xymatrix{\bullet \ar@{-}[r] & \bullet\ar@{-}[d]  \\ \bullet \ar@{..}[u]  \ar@{-}[ur] &  \bullet \ar@{..}[l] }\end{array}
} +
\Scale[\scale]{\begin{array}{c}
\xymatrix{\bullet \ar@{..}[r] & \bullet\ar@{..}[d] \\ \bullet \ar@{-}[u]  \ar@{-}[ur] &  \bullet \ar@{-}[l] }\end{array}
} +
\Scale[\scale]{\begin{array}{c}
\xymatrix{\bullet \ar@{-}[r] & \bullet\ar@{-}[d] \\ \bullet \ar@{-}[u]  \ar@{..}[ur] &  \bullet \ar@{..}[l] }\end{array}
} +
\Scale[\scale]{\begin{array}{c}
\xymatrix{\bullet \ar@{-}[r] & \bullet\ar@{-}[d]  \\ \bullet \ar@{..}[u]  \ar@{..}[ur] &  \bullet \ar@{-}[l]}\end{array}
} +
\Scale[\scale]{\begin{array}{c}
\xymatrix{\bullet \ar@{..}[r] & \bullet\ar@{-}[d]  \\ \bullet \ar@{-}[u]  \ar@{..}[ur] &  \bullet \ar@{-}[l]}\end{array}
} +
\Scale[\scale]{\begin{array}{c}
\xymatrix{\bullet \ar@{-}[r] & \bullet\ar@{..}[d] \\ \bullet \ar@{-}[u]  \ar@{..}[ur] &  \bullet \ar@{-}[l]}\end{array}
} \nonumber \\
& = \sum_{\mathcal{T}} \prod_{e \in \mathcal{T}} \frac{1}{g_e}.  \nonumber 
\end{align}
In the second line of both expressions we gave a representation of the determinant expansion in terms of spanning trees (weights are intended to be multiplied over solid edges of the diagram), which we compactly resume in the third line in terms of the spanning-tree polynomials found e.g. in Ref.\,\cite[Th.\,3.10]{nakanishi}, where $\mathcal{T}$ ranges over spanning trees. Then the identity ${\det}\; \Ll_0 / {\det}\; \Ga_1  = {\det}\; \Gg$ corresponds to a well-known duality of spanning tree polynomials, see e.g. Ref.\,\cite[ Eq.\,(4.11)]{sokal}.

Furthermore we can specialize this duality to eigenvalues as follows. Let
\begin{align}
\Ks_0 & = \Ga_0^{1/2} \Ll_0 \Ga_0^{1/2} = \left(\begin{array}{cc} 1 + \frac{g_2 + g_3}{g_4}  & - \frac{g_2}{\sqrt{g_4 g_5}} \\
-\frac{g_2}{\sqrt{g_4 g_5}} & 1 + \frac{g_1 + g_2}{g_5} \end{array}\right)  \nonumber \\  
\Sigma_1 & = \Ga_1^{1/2} \Ll_1 \Ga_1^{1/2} =  \left(\begin{array}{ccc} 1 + \frac{g_1}{g_5} & - \frac{\sqrt{g_1 g_2}}{g_5} & 0 \\ -\frac{\sqrt{g_1 g_2}}{g_5} & 1 + \frac{g_2}{g_4} + \frac{g_2}{g_5} & \frac{\sqrt{g_2 g_3}}{g_4} \\ 0 & \frac{\sqrt{g_2 g_3}}{g_4} & 1 + \frac{g_3}{g_4} \end{array}\right). \nonumber
\end{align}
Obviously $\det \Ks_0 = \det \Sigma_1$. It can be checked by direct substitution that the vector $(1/\sqrt{g_1}, 1/\sqrt{g_2},- 1/\sqrt{g_3})^\dagger$
is an eigenvector of $\Sigma_1$ relative to eigenvalue $\lambda=1$. The characteristic polynomials of the two matrices are given by
\begin{align}
\zeta_0(\lambda) & = \lambda^2 - \left(x + 2\right)\lambda +  y  \nonumber \\
\zeta_1(\lambda) & = \lambda^3 - \left(x +3 \right) \lambda^2 + \left(y + x + 2 \right) \lambda - y \nonumber
\end{align}
where
\begin{align}
y & = \frac{g_1 g_2 + g_1 g_3 + g_1 g_4 + g_4 g_2 + g_4 g_5 + g_5 g_2 + g_5 g_3 + g_2 g_3}{g_4 g_5} \nonumber \\
x & = \frac{g_2 g_5 + g_3 g_5 + g_1 g_4 + g_4 g_2}{g_4 g_5} \nonumber
\end{align}
Notice that $\zeta_1(\lambda) = (\lambda -1) \zeta_0(\lambda)$, therefore the two matrices have the same spectrum but for eigenvalue $1$. It can be checked that all other eigenvalues are larger than $1$.

\end{paragraph}

\section{Notation}

We consider a finite $n$-dimensional \Hilbert space $\Hil$ with nondegenerate scalar product $\HH:\Hil\times \Hil \to \C$.  Let $\Hil^\ast$ be its algebraic dual. The action of a 1-form $v^\ast$ on a vector $w$ is denoted $v^\ast[w]$, and, vice versa, vectors act linearly on 1-forms via $w[v^\ast]=v^\ast[w]$,  by virtue of the canonical isomorphism between $\Hil$ and the dual's dual $\Hil^{\ast\ast}$.

A bilinear form $A:\Hil \times \Hil \to \C$ induces a map $A:\Hil\to\Hil^\ast$ (denoted by the same symbol) from the \Hilbert space to its algebraic dual via $Av[\,\cdot\,] := A(\,\cdot\,,v)$. In particular, the scalar product induces a canonical isomorphism between vectors and linear forms, and a scalar product $\HH^{-1}:\Hil^\ast \times \Hil^\ast \to \C$ in the dual space, defined by $\HH^{-1}(\HH v,\HH w) := \HH(v,w)$.

Operators, i.e. linear maps from the \Hilbert space to itself, are adorned by a hat $\widehat{O}:\Hil \to \Hil$. Their adjoints $\widehat{O}^\dagger: \Hil \to \Hil$ are defined by $\HH(v,\widehat{O}w)=: \HH(\widehat{O}^\dagger v,w)$.

Given a bilinear form $A$, one can obtain an operator $\widehat{A} := \HH^{-1}  A: \Hil \to \Hil$. The pseudodeterminant $\det_+ A$ of a bilinear form is defined as the product of the nonvanishing eigenvalues of $\widehat{A}$. Furthermore, if $A$ is \Hermitian  positive-semidefinite, one can define the unique map $\sqrt{\widehat{A}} : \Hil \to \Hil^\ast$ such that $A(v,w) = H^{-1}(\sqrt{\widehat{A}} v, \sqrt{\widehat{A}} w)$.

Operators $\widehat{O}$ and bilinear forms $A$ are denoted by italic uppercase symbols, (the nonvanishing blocks of) their matrix representatives in a preferred basis by the corresponding roman characters $\widehat{\mathrm{O}},\mathrm{A}$.

\section{Setup}

We consider a Hilbert space edowed, on top of the scalar product, with another nondegenerate positive-definite \Hermitian form $G$ that we call the {\it metric}, which also induces an inverse metric $\GG^{-1}$ in the dual space, defined by $\GG(v,w) := \GG^{-1}(\GG v,\GG w)$. The scalar product will then play the role of a ``reference measure'' for the metric.

We now consider a decomposition $\Hil = \Hil_0 \oplus \Hil_1$ into nontrivial complementary  subspaces, with $n_0 =\dim \Hil_0 \neq 0$ and $n_1 = n-n_0$. Let $\PP_1:\Hil \to \Hil$ be a projection with range $\Hil_1$ and kernel $\Hil_0$, neither null nor the identity, and $\PP_0 = \II-\PP_1$ its complement with range $\Hil_0$ and kernel $\Hil_1$. In general, $\PP_0$ and $\PP_1$ are oblique, that is, not self-adjoint. Correspondingly, the dual space is decomposed into $\Hil^\ast = \Hil^\ast_0 \oplus \Hil^\ast_1$, where $\Hil^\ast_0$ is the space of all linear forms that vanish on $\Hil_1$ and $\Hil^\ast_1$ that of linear forms that vanish on $\Hil_0$. Obviously $\dim \Hil^\ast_0 = n_0 = \dim \Hil_0$. We then introduce oblique complementary projections $\PP^\ast_0$ and $\PP^\ast_1 = \II^\ast - \PP_1$ on $\Hil^\ast$, $\II^\ast$ being identity in $\Hil^\ast$. Obviously\footnote{\emph{Proof.} Decomposing $w = \PP_0w+\PP_1w$, one has $\PP_0^\ast v^\ast[w] =  \PP_0^\ast v^\ast[\PP_0 w]$, since by definition $\PP_0^\ast v^\ast$ vanishes on $\Hil_1$. Similarly, decomposing $v^\ast = \PP_0^\ast v^\ast+\PP_1^\ast v^\ast$ and considering that $\Hil^{\ast\ast} = \Hil$, one obtains $\PP_0^\ast v^\ast[w] = \PP_0^\ast v^\ast [\PP_0w] = v^\ast[\PP_0w]$, yielding $\HH(\HH^{-1}v^\ast,\PP_0w)=\HH(\HH^{-1}\PP_0^\ast v^\ast,w)$, $\forall v^\ast \in \Hil^\ast,w\in\Hil$, and the conclusion follows. }
$\PP_{0}^\ast  =  \HH \PP_{0} ^\dagger \HH^{-1}$.

Finally we define the main actors of this work.

\begin{definition}
The metrics $\LL_0,\LL_1:\Hil \times \Hil \to \C$ induced by $\GG$ respectively on $\Hil_0$ and on $\Hil_1$, and those $\GA_0,\GA_1:\Hil^\ast \times \Hil^\ast \to \C$ induced by $\GG^{-1}$ respectively on $\Hil^\ast_0$ and on $\Hil^\ast_1$, are defined by
\begin{align}
\begin{array}{l} \LL_0(v,w) := \GG(\PP_0 v,\PP_0w)  \\ 
\LL_1(v,w) := \GG(\PP_1 v,\PP_1 w)
\end{array}
, \qquad \forall\, v,w \in \Hil ,
\end{align}
and by
\begin{align}
\begin{array}{l} 
\GA_0(v^\ast,w^\ast) := \GG^{-1}(\PP_0^\ast v^\ast,\PP_0^\ast w^\ast) \\
\GA_1(v^\ast,w^\ast) := \GG^{-1}(\PP_1^\ast v^\ast,\PP_1^\ast w^\ast)\end{array}, \qquad \forall\, v^\ast,w^\ast \in \Hil^\ast.
\end{align}
\end{definition}

\begin{definition}
The forms $\KS_0,\KS_1,\SI_0,\SI_1:\Hil^\ast \times \Hil^\ast \to \C$ are defined as
\begin{equation}
\begin{array}{l}
\KS_0(v^\ast,w^\ast) := H^{-1}(\sqrt{\widehat{\LL_0 \GA_0}} \, v^\ast, \sqrt{\widehat{\LL_0 \GA_0}} \, w^\ast)  \\
\KS_1(v^\ast,w^\ast) := H^{-1}(\sqrt{\widehat{\LL_1 \GA_1}} \,v^\ast, \sqrt{\widehat{\LL_1 \GA_1}} \,w^\ast) \\
\SI_1(v^\ast,w^\ast) := H^{-1}(\sqrt{\widehat{\LL_0 \GA_0}}^\dagger \, v^\ast, \sqrt{\widehat{\LL_0\GA_0}}^\dagger \, w^\ast) \\
\SI_0(v^\ast,w^\ast) := H^{-1}(\sqrt{\widehat{\LL_0 \GA_0}}^\dagger \, v^\ast, \sqrt{\widehat{\LL_0\GA_0}}^\dagger \, w^\ast)
\end{array} , \qquad \forall\,  v^\ast,w^\ast \in \Hil^\ast,
\end{equation}
(where it is understood that $\sqrt{\widehat{AB}} = \sqrt{\widehat{A}} \sqrt{\widehat{B}}$).
\end{definition}

The above setup greatly simplifies in an orthonormal basis, in which case $\Hh = \widehat{\Ii} = \mathrm{diag}\,(1,1,\ldots,1)$,  $\dagger$ is matrix transposition and complex conjugation, $\Pp_0^\ast = \Pp_0^\dagger$, $\Pp_1^\ast = \Pp_1^\dagger$. For the induced metrics we have $\Ll_0 = \Pp_0^\dagger \Gg \Pp_0$, $\Ll_1 = \Pp_1^\dagger \Gg \Pp_1$, $\Ga_0 = \Pp_0 \Gg^{-1} \Pp_0^\dagger$, $\Ga_1 = \Pp_1 \Gg^{-1}\Pp_1^\dagger$, and futhermore $\Ks_0 = \Ga_0^{1/2} \Ll_0 \Ga_0^{1/2}$, $\Ks_1 = \Ga_1^{1/2} \Ll_1 \Ga_1^{1/2}$, $\Sigma_0 = \Ll_0^{1/2} \Ga_0 \Ll_0^{1/2}$ and $\Sigma_1 = \Ll_1^{1/2} \Ga_1 \Ll_1^{1/2}$ where the square roots are uniquely defined since Hermitian matrices have unique Hermitian square roots. 
However, it will be convenient to work in a different basis that gives a handy block-structure of matrices, where the scalar product is {\it not} the identity matrix; that's the reason we maintain this level of formality.


\section{Results}

The first theorem establishes a connection between the volume elements of the metrics induced by $\GG$ and $\GG^{-1}$ on the oblique subspaces. The second establishes a relation between the forms $\KS_0$ and $\Sigma_1$, and between $\KS_1$ and $\Sigma_0$ which, as a consequence, implies that they have the same spectra up to the multiplicity of $0$ and $1$ (and therefore the same singular values for operators $\sqrt{\widehat{\GA_0 \LL_0}}$ and $\sqrt{\widehat{\LL_1 \GA_1}}$, but for $0$ and $1$).

\begin{theorem}
The pseudodeterminants of the induced metrics are related by
\begin{align}
\frac{\det_+ \LL_1}{\det_+ \GA_0} = \frac{\det_+ \LL_0}{\det_+ \GA_1} = {\det}_+ \GG. \label{eq:statement}
\end{align}
\end{theorem}

\begin{proof}
Let $(v_i,w_j)$, be a basis for $\Hil$ such that the first $i=1,\ldots,n_0$ vectors are a basis for $\Hil_0$ and the last $j=n_0+1,\ldots,n$  are a basis for $\Hil_1$. We denote by $u_i$ a generic vector in this basis. We choose as dual basis the one $(v^\ast_i,w^\ast_j)$ defined by $u^\ast_i [u_j] = \delta_{ij}$, with $u = v,w$. In general, this basis cannot be orthonormal. However, it is possible to choose a basis (that we call {\it natural}) such that the first $n_0$ vectors are orthonormal among themselves and the last $n_1$ vectors are orthonormal among themselves, i.e. $H(v_i,v_j)= H(w_i,w_j) = \delta_{ij}$, while in general $\Omega_{ij} := H(v_i,w_j) \neq 0$.

Let us consider $\widehat{\GG} = \HH^{-1} \GG$ and $\widehat{\GG}^{-1} = \GG^{-1} \HH = \GG^{\ast} \HH$, which in the natural basis read
\begin{equation}
\begin{aligned}
\widehat{\Gg} & = \left(\begin{array}{cc} \Ii & \Omega^\dagger \\ \Omega &  \Ii  \end{array}\right)^{\!\!-1}  \left(\begin{array}{cc} \Ll_0 & \Vv^\dagger \\ \Vv &  \Ll_1  \end{array}\right) \label{eq:gg1} \\
\widehat{\Gg}^{-1} & = \left(\begin{array}{cc} \Ga_0 & \La^\dagger  \\ \La & \Ga_1 \end{array}\right) \left(\begin{array}{cc} \Ii & \Omega^\dagger \\ \Omega &  \Ii  \end{array}\right)
\end{aligned}
\end{equation}
where $\Ll_0,\Ll_1,\Ga_0,\Ga_1$ are the nonvanishing  \Hermitian square blocks of the matrix representatives of $\LL_0,\LL_1,\GA_0,\GA_1$, and $\Vv_{ij} := G(v_i,w_j)$, $\La_{ij} := G^\ast(v^\ast_i,w^\ast_j)$. From properties of inverses and determinants of partitioned block matrices \cite{bierens} we obtain
\begin{equation}
\begin{aligned}\label{eq:multi}
\Ga_0^{-1} & = \Ll_0  - \Vv^\dagger \Ll_1^{-1} \Vv \\
\Ga_1^{-1} & = \Ll_1  - \Vv \Ll_0^{-1} \Vv^\dagger   \\
\Ll_0^{-1} & = \Ga_0 - \La^\dagger \Ga_1^{-1} \La  \\
\Ll_1^{-1} & = \Ga_1 - \La \Ga_0^{-1} \La^\dagger 
\end{aligned}
\end{equation}
and the following expressions for the determinants
\begin{equation}
\begin{aligned}
\det\, \widehat{\Gg}  & =
\frac{\det\,  \Ll_0 }{\det\,  (\Ii - \Omega\Omega^\dagger ) \,  \det\,  \Ga_1} \\
\det\,  \widehat{\Gg}^{-1} & =
\frac{\det\, (\Ii - \Omega\Omega^\dagger ) \, \det\,  \Ga_0}{ \det \,   \Ll_1}.
\end{aligned}
\end{equation}
Finally we want to relate these matrix determinants to operator determinants. Clearly ${\det}_+ \GG = {\det}\,  \widehat{\Gg} $. As regards $\det_+ \LL_0$, we need to consider the eigenvalues of the operator $\widehat{\LL}_0 = \GG^{-1}  \LL_0$ which has matrix form
\begin{align}
\widehat{\Ll}_0 = \left(\begin{array}{cc} \Ii & \Omega^\dagger \\ \Omega &  \Ii  \end{array}\right)^{\!\!-1} \left(\begin{array}{cc} \Ll_0 & 0 \\ 0 & 0 \end{array}\right) =  \left(\begin{array}{cc} (\Ii - \Omega^\dagger \Omega)^{-1}  \Ll_0 & 0 \\ - (\Ii - \Omega \Omega^\dagger)^{-1} \Omega  \Ll_0 &0 \end{array}\right).  
\end{align}
The eigenvalue equation $\widehat{\LL}_0 v = \lambda v$ in the natural basis reads
\begin{align}
\left(\begin{array}{c} (\Ii - \Omega^\dagger \Omega)^{-1}  \Ll_0 v_0 \\ - (\Ii - \Omega \Omega^\dagger)^{-1} \Omega  \Ll_0 v_0 \end{array}\right) = \lambda \left(\begin{array}{c} v_0 \\ v_1 \end{array} \right).
\end{align}
The null eigenvalue corresponds to eigenvectors with $v_0 = 0$. For nonvanishing eigenvalues, the first equation yields the reduced eigenvalue equation $(\Ii - \Omega^\dagger \Omega)^{-1}  \Ll_0 v_0 = \lambda  v_0$ and the second returns the second part of the eigenvector as $v_1 =  - (\Ii - \Omega \Omega^\dagger)^{-1} \Omega (\Ii - \Omega \Omega^\dagger)v_0$. Consequently, we obtain ${\det}_+ \LL_0 = {\det}\,  \Ll_0  / {\det}\, (\Ii - \Omega^\dagger \Omega )$. By a similar reasoning, since
\begin{align}
\widehat{\Ga}_1 = \left(\begin{array}{cc} 0 & 0 \\ 0 & \Ga_1 \end{array}\right) \left(\begin{array}{cc} \Ii & \Omega^\dagger \\ \Omega &  \Ii  \end{array}\right) = \left(\begin{array}{cc} 0 & 0 \\  \Ga_1  \Omega &  \Ga_1 \end{array}\right), \label{eq:opga}
\end{align}
we obtain ${\det}_+\, \GA_1 = {\det}\, \Ga_1$, which completes the proof.
\end{proof}

Since the determinant of a metric $G$ is its volume element (measured in units of the volume element of $H$), then the above theorem establishes a fundamental symmetry between volume elements induced by the identification of nonorthogonal subspaces of a metric \Hilbert space.

When the metric and the scalar product coincide, $\GG = \HH$, it is well known \cite{szyld} that for complementary oblique projections
\begin{align}
\|\PP_0 \| = \|\PP_1 \|, 
\end{align} 
with the standard operator norm induced by the vector norm, and coinciding with the modulus of the largest singular value. This norm identity is actually a corollary of an even stronger result, since it can be proven that $\PP_0$ and of $\PP_1$ have the same singular values (eigenvalues of $\PP_0^\dagger \PP_0$ and $\PP_1 \PP_1^\dagger$, up to the multiplicity (possibly vanishing) of $0$ and $1$ \cite{lewko}. It is the objective of the next theorem to prove an analogous result in our generalized context. 

\begin{theorem}
The forms $\KS_0$, $\KS_1$, $\SI_0$ and $\SI_1$ have the same spectrum, up to the multiplicity of eigenvalues $0$ and $1$. All nonvanishing eigenvalues are not smaller than $1$.
\end{theorem}

\begin{proof}

In the first part of the proof we will show, as intuitive, that it all boils down to considering the eigenvalues of the matrices $\Ll_0^{1/2} \Ga_0 \Ll_0^{1/2}$, $\Ll_1^{1/2} \Ga_1 \Ll_1^{1/2}$, $\Ga_0^{1/2} \Ll_0 \Ga_0^{1/2}$, and $\Ga_1^{1/2} \Ll_1 \Ga_1^{1/2}$, in the second part we derive the result.

Let us first find the matrix representatives of $\sqrt{\widehat{\LL_0}}$, $\sqrt{\widehat{\GA_0}}$, $\sqrt{\widehat{\LL_1}}$ and $\sqrt{\widehat{\GA_1}}$ in the natural basis. We have by definition
\begin{equation}
\begin{aligned}
\left(\begin{array}{cc} \Ll_0 & 0 \\ 0 & 0 \end{array}\right) & = \left(\begin{array}{cc} \sqrt{\widehat{\Ll_0}}^\dagger & 0 \\ 0 & 0 \end{array}\right) 
 \left(\begin{array}{cc} \Ii & \Omega^\dagger \\ \Omega &  \Ii  \end{array}\right)^{\!\!-1}
 \left(\begin{array}{cc} \sqrt{\widehat{\Ll_0}} & 0 \\ 0 & 0 \end{array}\right) \\
 \left(\begin{array}{cc} 0 & 0 \\ 0 & \Ll_1 \end{array}\right) & = \left(\begin{array}{cc} 0 & 0 \\ 0 & \sqrt{\widehat{\Ll_1}}^\dagger \end{array}\right) 
 \left(\begin{array}{cc} \Ii & \Omega^\dagger \\ \Omega &  \Ii  \end{array}\right)^{\!\!-1}
 \left(\begin{array}{cc} 0 & 0 \\ 0 & \sqrt{\widehat{\Ll_1}} \end{array}\right) \\
 \left(\begin{array}{cc} \Ga_0 & 0 \\ 0 & 0 \end{array}\right) & = \left(\begin{array}{cc}  \sqrt{\widehat{\Ga_0}}^\dagger & 0 \\ 0 & 0 \end{array}\right) 
 \left(\begin{array}{cc} \Ii & \Omega^\dagger \\ \Omega &  \Ii  \end{array}\right)
 \left(\begin{array}{cc} \sqrt{\widehat{\Ga_0}} & 0 \\ 0 & 0 \end{array}\right) \\
 \left(\begin{array}{cc} 0 & 0 \\ 0 & \Ga_1 \end{array}\right) & = \left(\begin{array}{cc} 0 & 0 \\ 0 & \sqrt{\widehat{\Ga_1}}^\dagger  \end{array}\right) 
 \left(\begin{array}{cc} \Ii & \Omega^\dagger \\ \Omega &  \Ii  \end{array}\right)
 \left(\begin{array}{cc} 0 & 0 \\ 0 & \sqrt{\widehat{\Ga_1}}  \end{array}\right)
\end{aligned}
\end{equation}
yielding
\begin{equation}
\begin{aligned}
\sqrt{\widehat{\Ll_0}} & = (\Ii - \Omega^\dagger \Omega)^{1/2}   \Ll_0^{1/2} \\
\sqrt{\widehat{\Ll_1}} & = (\Ii - \Omega \Omega^\dagger)^{1/2}   \Ll_1^{1/2} \\
\sqrt{\widehat{\Ga_0}} & =  \Ga_0^{1/2} \\
\sqrt{\widehat{\Ga_1}} & =  \Ga_1^{1/2}.
\end{aligned}
\end{equation}
Imposing
\begin{equation}
\begin{aligned}
\left(\begin{array}{cc} \Ks_0 & 0 \\ 0 & 0 \end{array}\right) & = \left(\begin{array}{cc} \sqrt{\widehat{\Ga_0}}^\dagger   \sqrt{\widehat{\Ll_0}}^\dagger & 0 \\ 0 & 0 \end{array}\right) 
 \left(\begin{array}{cc} \Ii & \Omega^\dagger \\ \Omega &  \Ii  \end{array}\right)^{\!\!-1}
 \left(\begin{array}{cc} \sqrt{\widehat{\Ll_0}} \sqrt{\widehat{\Ga_0}} & 0 \\ 0 & 0 \end{array}\right) \\
 \left(\begin{array}{cc} 0 & 0 \\ 0 & \Ks_1 \end{array}\right) & = \left(\begin{array}{cc} 0 & 0 \\ 0 & \sqrt{\widehat{\Ga_1}}^\dagger \sqrt{\widehat{\Ll_1}}^\dagger    \end{array}\right) 
 \left(\begin{array}{cc} \Ii & \Omega^\dagger \\ \Omega &  \Ii  \end{array}\right)^{\!\!-1}
 \left(\begin{array}{cc} 0 & 0 \\ 0 & \sqrt{\widehat{\Ll_1}}  \sqrt{\widehat{\Ga_1}} \end{array}\right),
\end{aligned}
\end{equation}
we find the intuitive expressions
\begin{equation}
\begin{aligned}
\Ks_0 & = \Ga_0^{1/2} \Ll_0 \Ga_0^{1/2} \\
\Ks_1 & = \Ga_1^{1/2} \Ll_1 \Ga_1^{1/2}.
\end{aligned}
\end{equation}
Since $\widehat{\Ks}_0$ and $\widehat{\Ks}_1$ are defined by equations analogous to Eq.\,(\ref{eq:opga}), then their spectra coincide with those of $\Ks_0$, $\Ks_1$. We can proceed in an analogous way for $\widehat{\Sigma}_0$ and $\widehat{\Sigma}_1$, finding
\begin{equation}
\begin{aligned}
\widehat{\mathrm{\Sigma}}_0 & = (1+\Omega^\dagger \Omega)^{1/2} \Sigma_0 (1+\Omega^\dagger \Omega)^{-1/2} \\
\widehat{\mathrm{\Sigma}}_1 & = (1+\Omega \Omega^\dagger)^{1/2} \Sigma_1 (1+\Omega \Omega^\dagger)^{1/2},
\end{aligned}
\end{equation}
where
\begin{equation}
\begin{aligned}
\Sigma_0 & := \Ll_0^{1/2} \Ga_0 \Ll_0^{1/2} \\
\Sigma_1 & := \Ll_1^{1/2} \Ga_1 \Ll_1^{1/2}.
\end{aligned}
\end{equation}
Clearly $\Sigma_0$ and $\Ks_0$ have the same spectrum up to the multiplicity of eigenvalue $0$, and so do $\Sigma_1 $ and $\Ks_1$.

Now consider the Eqs.(\ref{eq:multi}), that we rewrite as
\begin{equation}\label{eq:lalala}
\begin{aligned}
\Sigma_0^{-1} & = \Ii - \mathrm{B}_0^\dagger \Ks_{1}^{-1} \mathrm{B}_0 \\
\Ks_0^{-1} & = \Ii  - \mathrm{D}_0^\dagger \Sigma_1^{-1}\mathrm{D}_0 \\
\Sigma_1^{-1}& = \Ii - \mathrm{D}_1^\dagger \Ks_0^{-1}\mathrm{D}_1 \\
\Ks_{1}^{-1} & = \Ii - \mathrm{B}_1^\dagger \Sigma_{0}^{-1} \mathrm{B}_1
\end{aligned}
\end{equation}
where
\begin{equation}
\begin{aligned}
\mathrm{B}_0 & := \Ga_1^{1/2} \Vv  \Ll_0^{-1/2}  \\
\mathrm{B}_1 & := \Ll_0^{1/2} \La^\dagger  \Ga_1^{-1/2} \\
\mathrm{D}_0 & := \Ga_0^{1/2} \Vv^\dagger \Ll_1^{-1/2}\\
\mathrm{D}_1 & := \Ll_1^{1/2} \La \Ga_0^{-1/2}
\end{aligned}
\end{equation}
Notice that by taking the product $\widehat{\Gg} \widehat{\Gg}^{-1} = \widehat{\Ii}$ from Eqs.(\ref{eq:gg1}), we obtain with a few simple manipulations 
\begin{equation}
\begin{aligned}
\mathrm{B}_0= -  \mathrm{B}_1^\dagger  =: \mathrm{B} \\
\mathrm{D}_0= -  \mathrm{D}_1^\dagger  =: \mathrm{D}
\end{aligned}
\end{equation}
and furthermore
\begin{align}
\mathrm{D}^\dagger \mathrm{D}  = \sqrt{\widehat{ \Ga_1 \Ll_1}}^{\;-1} \mathrm{B}\mathrm{B}^\dagger  \sqrt{\widehat{ \Ga_1 \Ll_1}}. \label{eq:similarity}
\end{align}
We then obtain from the second and third of Eqs.\,(\ref{eq:lalala})
\begin{equation}
\begin{aligned}
\Ks_0^{-1} & = \Ii  - \mathrm{D}^\dagger \Sigma_1^{-1}\mathrm{D}   \\
\Sigma_1^{-1}& = \Ii - \mathrm{D} \Ks_0^{-1}\mathrm{D}^\dagger .
\end{aligned}
\end{equation}
and similarly for $\Ks_1^{-1}$ and $\Sigma_0^{-1}$ in terms of $\mathrm{B}$. Replacing the first equation into the second, we obtain the recursive relations
\begin{equation}
\begin{aligned}
\Ks_0^{-1} & = \Ii  -  \mathrm{D}  \mathrm{D}^\dagger +  \mathrm{D}  \mathrm{D}^\dagger \Ks_0^{-1} \mathrm{D}  \mathrm{D}^\dagger , \\
\Sigma_1^{-1} & = \Ii  -   \mathrm{D}^\dagger \mathrm{D}  + \mathrm{D}^\dagger   \mathrm{D}  \Sigma_1^{-1} \mathrm{D}^\dagger \mathrm{D}  .
\end{aligned}
\end{equation}
The latter equations are solved by
\begin{equation}\label{eq:kssol}
\begin{aligned}
\Ks_0 & = \Ii + \mathrm{D}\mathrm{D}^\dagger \\
\Sigma_1 & = \Ii +  \mathrm{D}^\dagger \mathrm{D}
\end{aligned}
\end{equation}
as can be found by direct replacement (notice that $\Ks_0$ and $\Sigma_1$ are unquely defined by construction). Letting $v$ be an eigenvector of $\Ks_1$ relative to eigenvalue $\lambda$, we find
\begin{align}
\mathrm{D}^\dagger \mathrm{D} v = (\lambda -1) v.
\end{align}
Since the left-hand side matrix is positive-semidefinite, clearly all eigenvalues of $\Ks_0$ and $\Sigma_1$ must be not smaller than $1$. Acting with $\mathrm{D}$ on the latter expression we obtain that $\mathrm{D} v$ is an eigenvector of $\Ks_0 $ with respect to the same eigenvalue. Let $n_0 \geq n_1$, without loss of generality, and let $r \leq n_1$ be the rank of $\mathrm{D}$. We conclude that $\KS_0$ [resp. $\Sigma_1$] has $r$ positive eigenvalues strictly larger than $1$, eigenvalue $1$ with multiplicity $n_0-r$ [resp. $n_1 - r$] and eigenvalue $0$ with multiplicity $n_1$ [resp. $n_0$], and that the eigenvalues larger than $1$ are the same. Same applies to $\Ks_1$ and $\Sigma_0$ given the similarity Eq.\,(\ref{eq:similarity}).
\end{proof}

\section{Relation to graph polynomials}

Here we report on a relations between the above results and known properties of the spanning-tree polynomial in graph theory. We build upon the results of Ref.\,\cite{polettini1}, to which we refer for further details. 

By the spectral theorem we can pick an orthonormal basis $E = \{e_i\}_{i = 1}^n$ that makes the metric diagonal $\Gg = \mathrm{diag}\{g_i\}_{i =1}^n$. We call such basis the {\it edge space}. The diagonal entries $g_i$ can be seen as positive weights associated to each edge. We introduce the vertex set $X$ of vertices $x$ and a map $\delta : E \to X \times X$ associating an ordered pair of vertices to each edge, one (denoted $\stackrel{e_i}{\rightarrow} x $) being the target and the other (denoted $\stackrel{e_i}{\leftarrow} x$) being the origin of the edge. The choice of which $x$ is the origin and which is the target fixes a completely arbitrary orientation of the graph, on which the results below do not depend. We can make this map into a linear operator $\delta : \Hil \to \Hil_X$ mapping arbitrary linear combinations of edges into the linear space $\Hil_X$ generated by vertices. The operator acts on basis vectors according to
\begin{align}
\delta_{x,i} = \left\{\begin{array}{ll}
+1, & \mathrm{if}\, \stackrel{e_i}{\rightarrow} x \\ 
-1, & \mathrm{if}\, \stackrel{e_i}{\leftarrow} x \\
0, & \mathrm{otherwise}
\end{array}\right.  .
\end{align}
The quadruple $\mathcal{G}= (E,X,\delta,\Gg)$ forms an oriented weighted graph. We assume that the graph is connected, in the sense that for any two complementary subsets of edges $E_1$, $E_2 = E \setminus E_1$, the corresponding sets of boundary vertices intersect (so that there is a path between any two boundary vertices) and their union is $X$ (so that there are no isolated vertex). Under this hypothesis one has  $|X| \leq |E|+1$; the number $|C|= |E|-|X|+1 \geq 0$ is called the {\it cyclomatic number}. We assume (as customary) that the graph does not include loops, that is, edges whose boundary vertices coincide.

We stipulate that $\Hil_0$ is the kernel of $\delta$, called the {\it cycle space}, with dimension $n_0 = |C|$, and that $\Hil_1^\ast$ is the image of $\delta$, called the {\it cocycle space}, with dimension $n_1 = |X|-1$. A basis for the cycle and for the cocycle spaces is found by the following procedure. We pick an arbitrary spanning tree, i.e. a set of $|X|-1$ edges that connects all vertices. Let the {\it cochords} $\{e_\mu\}_{\mu = 1}^{n_1}$  be the (vector representatives) of the edges that belong to the tree, and the {\it chords} $\{e_\alpha\}_{\alpha = n_1+1}^{n}$ be the (co-vector representatives) of the edges that do not belong to the tree. Adding a chord $e_\alpha$ to the spanning tree identifies a unique basis cycle vector $c_\alpha$. Removing a cochord from the spanning tree identifies a unique cocycle co-vector $c_\mu$. One has
\begin{align}
c_\mu [ c_\alpha]= e_\mu[e_\alpha] = 0, \qquad c_\mu[e_{\mu'}] = \delta_{\mu,\mu'}, \qquad c_\alpha[e_{\alpha'}] = \delta_{\alpha,\alpha'}.
\end{align}
The crucial result exposed in Ref.\,\cite{polettini1} is that chords span $\Hil_0^\ast$ and cochords span $\Hil_1$, and more precisely that, letting $\otimes$ be the outer product of a vector and a linear form, the two operators
\begin{equation}
\begin{aligned}
\PP_0 = \sum_{\alpha = 1}^{|C|} c_\alpha \otimes e_\alpha, \qquad
\PP_1 = \sum_{\mu = |C|+1}^{|E|} e_\mu \otimes c_\mu,
\end{aligned}
\end{equation}
are complementary projections, typically obique (except very special cases).

The induced metrics then read
\begin{equation}
\begin{aligned}
\LL_{0} & =  \sum_{\alpha,\alpha' = 1}^{|C|} (\Ll_0)_{\alpha,\alpha'}  e_\alpha \otimes e_{\alpha'}, & 
\LL_{1} & = \sum_{\mu,\mu' = 1}^{|V-1|} (\Ll_0)_{\mu,\mu'}   c_\mu \otimes c_{\mu'} \\
\GA_{0} & = \sum_{\alpha,\alpha' = 1}^{|C|} (\Ga_0)_{\alpha,\alpha'} c_\alpha \otimes c_{\alpha'},  & 
\GA_{1} & = \sum_{\mu,\mu' = 1}^{|V-1|} (\Ga_1)_{\mu,\mu'} e_\mu \otimes e_{\mu'}.
\end{aligned}
\end{equation}
where $ (\Ll_0)_{\alpha,\alpha'}  = G[c_\alpha,c_{\alpha'}]$, $(\Ll_0)_{\mu,\mu'}  =  G[e_\mu,e_{\mu'}]$, $ (\Ga_0)_{\alpha,\alpha'}  =  G^{-1}[e_\alpha,e_{\alpha'}]$, and $ (\Ga_1)_{\mu,\mu'}  = G^{-1}_{\mu,\mu'}[c_\mu,c_{\mu'}]$. That $\det \Ll_1 / \det \Ga_0 = \det \Gg = \prod_i g_i$ is obvious from the fact that they are diagonal matrices covering all the edges. Instead, the determinants of the cycle and cocycle overlap matrices $\Ll_0$ and $\Ga_1$ are well-known to give the spanning tree/cotree polynomials described in the opening example, see Th.\,3.10 in Ref.\,\cite{nakanishi} (see also \cite{marcolli}), and $\det \Ll_0/\det \Ga_1  = \det \Gg = \prod_i g_i$ corresponds to Eq.\,(4.11) in the review paper Ref.\,\cite{sokal}, that relates a the Tutte polynomial of a planar graph and its dual $\mathcal{G}^\ast$.

Now let $e^\ast_i[e_j] = \delta_{i,j}$ and let us introduce the matrix
\begin{align}
\mathrm{D}_{\mu,\alpha} = \sqrt{\frac{g_\mu}{g_\alpha}}\,  c_\mu[e^\ast_\alpha] = - \sqrt{\frac{g_\mu}{g_\alpha}}\,  e^\ast_\mu[c_\alpha].
\end{align}
The second identity follows from \cite[Theorem 3]{polettini1}. 
Finally, we have a particularly simple representation for $\Ks_0$ and $\Sigma_1$:
\begin{equation}
\begin{aligned}
(\Ks_0)_{\alpha,\alpha'} & = \frac{(\Ll_0)_{\alpha,\alpha'}}{\sqrt{g_{\alpha} g_{\alpha'}}} & & = \delta_{\alpha,\alpha'} + \sum_{\mu=1}^{n_1} \mathrm{D}_{\mu,\alpha} \mathrm{D}_{\mu,\alpha'} \\
(\Sigma_1)_{\mu,\mu'} & = \sqrt{g_{\mu} g_{\mu'}} (\Ga_1)_{\mu,\mu'} & & = \delta_{\mu,\mu'} + \sum_{\alpha=n_1+1}^{n} \mathrm{D}_{\mu',\alpha} \mathrm{D}_{\mu,\alpha}.
\end{aligned}
\end{equation}
These latter relations follow from the fact that the only chord belonging to cycle $c_\alpha$ is $e_\alpha$, which is accounted for by the Kroenecker delta, while cochords are enough to span the rest of a cycle and the intersections between two cycles $c_\alpha$ and $c_{\alpha'}$. Similarly for cocycles. We have thus reproduced by direct computation the result Eq.\,(\ref{eq:kssol}).

\bigskip
{\bf Acknowledgment.} The research was supported by the National Research Fund Luxembourg (core project THERMOCOMP) and by the European Research Council (ERC-2015-CoG Agreement No. 681456).


\end{document}